\documentclass[12pt,a4paper]{article}
\usepackage[left=2cm, right=2cm, top=2.5cm, bottom=2.5cm]{geometry}
\usepackage{enumerate}
\usepackage{amssymb}
\usepackage{amsmath}
\usepackage{mathrsfs}
\usepackage{amsthm}
\usepackage{mathtools}
\usepackage{enumitem}
\theoremstyle{definition}

\newtheorem{theorem}{Theorem}[section]
\newtheorem{lemma}[theorem]{Lemma}

\newtheorem{remark}[theorem]{Remark}

\usepackage[colorlinks,linkcolor=blue]{hyperref}
\numberwithin{equation}{section}

\usepackage{enumerate}

\allowdisplaybreaks[4]

\makeatletter

\newcommand{\Rmnum}[1]{\expandafter\@slowromancap\romannumeral #1@}
\makeatother

\begin{document}
	
	\title{Sobolev inequalities involving 2-tensor fields in manifolds with nonnegative sectional curvature}
	
	\author{
		Jie Wang 
		\thanks{E-mail: wangjie@whu.edu.cn, School of Mathematics and Statistics, Wuhan University, Wuhan, 430072, P.R.China. }
	}
	
	\date{}
	
	\maketitle
	
	\begin{abstract}
\noindent 	By applying the Alexandrov–Bakelman–Pucci (ABP) method, we establish both Log-Sobolev type inequality and Michael-Simon Sobolev inequality for smooth symmetric uniformly positive definite $\left( 0,2\right)$-tensor fields in manifolds with nonnegative sectional curvature. Our results extends the work in Yi-Zheng (Chinese Ann. Math. Ser. B, 45(3): 487–496. 2024) \cite{YZ} and Pham (Int. Math. Res. Not. IMRN, 6: 4656–4688. 2024) \cite{P}.
\\

	\end{abstract}

	\section{Introduction}
	
In 2021, Brendle \cite{B1} made a great breakthrough in isoperimetric problem, where he established an elegant Michael-Simon Sobolev inequality which holds on submanifolds in Euclidean
space of arbitrary dimension and codimension. When the codimension is at most 2, this solves a long-standing conjecture that the
(sharp) isoperimetric inequality holds for minimal submanifolds in $\mathbb{R}^{n+1}$. Later, Brendle \cite{B3} extended the Michael-Simon type inequality as well as the isoperimetric inequality to minimal submanifolds in general ambient spaces with nonnegative sectional curvature. Moreover, Brendle \cite{B2} also proved a sharp Log-Sobolev inequality for submanifolds in Euclidean space.

The most significant feature in these works is the development of the so-called Alexandrov-Bakelman-Pucci technique. By the same method, Brendle’s works on Sobolev inequality or Log-Sobolev inequality have been extended to several different curvature settings by other researchers. Yi and Zheng \cite{YZ} obtained a Log-Sobolev inequality which holds for closed submanifolds in Riemannian manifolds with nonnegative sectional curvature. Later, this result was further improved to manifolds with asymptotically nonnegative sectional curvature in \cite{DLL2} and asymptotically nonnegative intermediate Ricci curvature in \cite{LR}. For Sobolev inequality, Johne \cite{J} considered the case of nonnegative Bakry-Émery Ricci curvature, Dong-Lin-Lu \cite{DLL1} considered the case of asymptotically nonnegative curvature, and Ma-Wu \cite{MW} considered the asymptotically nonnegative intermediate Ricci curvature condition.

	In 2018, Serre \cite{S} proved a new sharp Sobolev type inequality involving positive symmetric matrix-valued
functions in $\mathbb{R}^{n}$ and showed some applications to fluid dynamics. Pham \cite{P} removed the convexity condition of domain in Serre's result and obtained a Sobolev inequality involving a positive symmetric matrix-valued function on a smooth bounded domain in $\mathbb{R}^{n}$. 

Inspired by \cite{P}, some results on Michael-Simon type Sobolev inequality for symmetric positive definite $(0, 2)$-tensor fields have been established recently in \cite{WY} and \cite{WYZ}.

In this paper, by using the ABP method, we will present a new Log-Sobolev type inequality and a Michael-Simon Sobolev inequality for smooth symmetric uniformly positive definite $\left( 0,2\right)$-tensor fields in Riemannian manifolds with nonnegative sectional curvature.
	
	Let $(M,g)$ be a complete noncompact Riemannian manifold of dimension $l$ with nonnegative Ricci curvature. Denote by $\left| B^{l}\right| $ the volume of the unit ball in $\mathbb{R}^{l}$.
	The asymptotic volume ratio of $M$ is defined as
	\begin{equation*}  
		\theta:=\underset{r\rightarrow\infty} {\lim}\dfrac{\left\lbrace p\in M: d(p, q)\leq r\right\rbrace }{|B^{l}|r^{l}}
	\end{equation*}
	where $q$ is some fixed point on the manifold. According to the Bishop-Gromov volume comparison theorem, the limit exists and that $\theta\leq1$.
    
    One of our main results is the following theorem which extends Theorem 1.1 in \cite{YZ} to the $\left( 0,2\right)$-tensor setting.

    \begin{theorem}	\label{main1}
		Let $M$ be a complete, noncompact Riemannian manifold of dimension $n+m$ with nonnegative sectional curvature and Euclidean volume growth (i.e., $\theta>0$).
		Suppose that $\Sigma^{n}$ is a closed submanifold of $M^{n+m}$.
		Then for any smooth symmetric uniformly positive definite $\left( 0,2\right)$-tensor field $A$ on $\Sigma$, we have
		\begin{equation}  \label{logsobolev}
\begin{aligned}
			&\int_{\Sigma}\left(\operatorname{det}A\right)^{\frac{1}{n}}\left(\log \left(\operatorname{det}A\right)^{\frac{1}{n}} +n +\frac{n}{2}\log(4\pi)+\log\theta\right)dV(x)-\int_{\Sigma}\frac{|\operatorname{div_{\Sigma} }A|^{2}}{\left(\operatorname{det}A\right)^{\frac{1}{n}}}dV(x)\\
			&-\int_{\Sigma}\frac{|\left\langle A,\Rmnum{2}\right\rangle |^{2}}{\left(\operatorname{det}A\right)^{\frac{1}{n}}}dV(x)\leq \int_{\Sigma}\left(\operatorname{det}A\right)^{\frac{1}{n}}dV(x)\left(\log \int_{\Sigma}\left(\operatorname{det}A\right)^{\frac{1}{n}}dV(x)\right),
\end{aligned}
		\end{equation}
		where $\Rmnum{2}$ denotes the second fundamental form of $\Sigma$.
	\end{theorem}
    
    \begin{remark}
   When $A$ is a conformal metric on $\Sigma$, i.e., $A=fg_{\Sigma}$ for some positive function $f$ on $\Sigma$, we have $\operatorname{div}_{\Sigma}A=df$, $\operatorname{det}A=f^{n}$, $A(\nu)=f$ and $\left\langle A,\Rmnum{2}\right\rangle=fH$, where $H$ is the mean curvature vector of $\Sigma$. In this special case, the above inequality is exactly the
results in \cite{YZ}.
    \end{remark}

   Another main result extends Theorem 1.8 in \cite{P} from Euclidean ambient space to the case that the ambient manifold has nonnegative sectional curvature.

	\begin{theorem}	\label{main2}
		Let $M$ be a complete, noncompact Riemannian manifold of dimension $n$ with nonnegative sectional curvature.
		Suppose that $D$ is a compact domain in $M$ with boundary $\partial D$.
		Then for any smooth symmetric uniformly positive definite $\left( 0,2\right)$-tensor field $A$ on $D$, we have
		\begin{equation} \label{sobolev}
			\int_{D}|\operatorname{div}A|+\int_{\partial D}A(\nu)\geq 
    		n\left| B^{n}\right|^{\frac{1}{n}}\theta^{\frac{1}{n}}
    		\left( \int_{\Sigma}\left( \mathrm{det}A\right)^{\frac{1}{n-1}}\right) ^{\frac{n-1}{n}},
		\end{equation}
		where $\nu$ is the unit outer normal vector field on $\partial D$.
	\end{theorem}

To end the introduction, we want to say some words on the motivation of
our results. When studying geometry of submanifolds, the Newton tensor of the second fundamental form is an important symmetric $(0,2)$-tensor which  relates with higher order mean curvature. Moreover, it is divergence-free.
We believe that one can use the Newton tensor in the Sobolev inequalities for 2-tensor fields to gain some interesting information about submanifolds.

    The paper is organized as follows. In Section 2, we recall some basic definitions and properties for $2$-tensors and asymptotic volume ratio. In Section 3, we provide the proof of Theorem \ref{main1}. In Section 4, we give the proof of Theorem \ref{main2}.
    
	\section{Preliminaries}
    Let $(M, g)$ be a complete noncompact $N=(n+m)$-dimension manifold. Let $\Sigma$ be a compact $n$-dimensional submanifold of $M^{n+m}$ with smooth boundary $\partial\Sigma$ (possibly  $\partial\Sigma=\varnothing$). Denote by $\nabla$ the Levi-Civita connection on $(M, g)$, and by $R$ the Riemann curvature tensor of $(M, g)$. Assume that $g_{\Sigma}$ is the induced metric on $\Sigma$, $\nabla^{\Sigma}$ is the Levi-Civita connection on $\Sigma$.
    
    The \emph{asymptotic volume ratio} $\theta$ can be regarded as the ratio of the volume of geodesic ball in $M$ to the volume of Euclidean ball in $\mathbb{R}^{N}$ with same, arbitrary large radius. Precisely, the
asymptotic volume ratio $\theta$ is defined as
	\begin{equation*}  
		\theta:=\underset{r\rightarrow\infty} {\lim}\dfrac{\left\lbrace p\in M: d(p, q)\leq r\right\rbrace }{|B^{N}|r^{N}},
	\end{equation*}
    where $q\in M$ is some fixed point and $ B^{N} $ is the unit ball in $\mathbb{R}^{N}$. If $M$ has nonnegative Ricci curvature, then Bishop-Gromov volume comparison theorem indicates $\theta$ exists and  $\theta\leq1$.  
    
    Let $A$ be a smooth symmetric uniformly positive definite $\left( 0,2\right)$-tensor field on $\Sigma$.
    For each point $x\in\Sigma$, we denote by $T_{x}\Sigma$ and $T_{x}^{\bot} \Sigma$ the tangent and normal space to $\Sigma$ at $x$, respectively.
    Let $\left( x^{1}, \dots, x^{n}\right) $ be a local coordinate system on $\Sigma$,
    the divergence of $A$ on $\Sigma$ is defined by
    \begin{equation*}
    	\mathrm{div}_{\Sigma}A:=
    	g_{\Sigma}^{ki}\nabla^{\Sigma}_{k}A_{ij}dx^{j}.
    \end{equation*}
    Let $T$ and $S$ be two $\left(0,2\right)$-tensor fields on $\Sigma$. In general, the inner product of $T$ and $S$ can be written as
    \begin{equation*}
    	\left\langle T, S\right\rangle 
    	=g_{\Sigma}^{ik}g_{\Sigma}^{jl}T_{ij}S_{kl}
    	=T_{ij}S^{ij}.
    \end{equation*}
    The composition of $T$ and $S$ is the $\left(0,2\right)$-tensor $T\circ S$ defined by 
    \begin{equation*}
    (T\circ S)_{ij}=g^{kl}_{\Sigma}T_{ik}S_{lj}.
    \end{equation*}
    The $\mathrm{det}T$ signifies the determinant of $T$, 
    which is defined by the determinant of $\left(1,1\right)$-tensor $g_{\Sigma}^{ik}T_{jk}\frac{\partial}{\partial x^{i}}\otimes dx^{j}$. 
    When $T\circ S=g_{\Sigma}$, 
    we refer to $T$ as the inverse tensor of $S$ denoted by $T^{-1}$.
    Meanwhile, 
    $\Rmnum{2}$ denotes the second fundamental form of $\Sigma$ as defined by
     \begin{equation*}
     	\left\langle \Rmnum{2}\left( X,Y\right), Z\right\rangle :=
     	\left\langle \nabla_{X}Y, \nu\right\rangle =
     	-\left\langle \nabla_{X}\nu, Y\right\rangle,
     \end{equation*}
    where $X$ and $Y$ are tangent vector fields on $\Sigma$,
    $\nu$ is a normal vector field to $\Sigma$. 
    Further, $\left\langle A, \Rmnum{2}\right\rangle\left( x\right) $ is the normal vector at $x\in\Sigma$ defined by
    \begin{equation*}
    	\left\langle A, \Rmnum{2}\right\rangle \left( x\right) =
         g_{\Sigma}^{ik}g_{\Sigma}^{jl}A_{ij}\Rmnum{2}(\frac{\partial}{\partial x ^{k}},\frac{\partial}{\partial x ^{l}}).
    \end{equation*}

	\section{Proof of Theorem \ref{main1}}
    In this section, we prove the Log-Sobolev inequality for 2-tensor fields. Let $M$ be a complete, noncompact Riemannian manifold of dimension $n+m$ with nonnegative sectional curvature and Euclidean volume growth.
		Let $\Sigma^{n}$ be a compact submanifold of $M^{n+m}$ without boundary and $A$ be any given smooth symmetric uniformly positive definite $\left( 0,2\right)$-tensor field on $\Sigma$.

   We first consider the special case that $\Sigma$ is connected, which is needed for the existence of solutions to a certain PDE. It is easy to check that \eqref{logsobolev} is scaling invariant about $A$. Hence we may assume that
   \begin{equation}  
			\int_{\Sigma}\left(\operatorname{det}A\right)^{\frac{1}{n}}\log \left(\operatorname{det}A\right)^{\frac{1}{n}} -\int_{\Sigma}\frac{|\operatorname{div }A|^{2}}{\left(\operatorname{det}A\right)^{\frac{1}{n}}}-\int_{\Sigma}\frac{|\left\langle A,\Rmnum{2}\right\rangle |^{2}}{\left(\operatorname{det}A\right)^{\frac{1}{n}}}=0,
		\end{equation}
		and the goal becomes to prove 
		       \begin{equation}\label{goal}
       \int_{\Sigma}\left(\operatorname{det}A\right)^{\frac{1}{n}}\left(n+\frac{n}{2}\log(4\pi)+\log\theta\right)\leq \int_{\Sigma}\left(\operatorname{det}A\right)^{\frac{1}{n}}\log\left( \int_{\Sigma}\left(\operatorname{det}A\right)^{\frac{1}{n}}\right).
       \end{equation}
   
   In order to apply the ABP method, we consider a certain partial differential equation as follows:
      \begin{equation}  \label{auxilrayPDE}
			\operatorname{div }_{\Sigma}(A(\nabla^{\Sigma}u))=\left(\operatorname{det}A\right)^{\frac{1}{n}}\log \left(\operatorname{det}A\right)^{\frac{1}{n}} -\frac{|\operatorname{div }A|^{2}}{\left(\operatorname{det}A\right)^{\frac{1}{n}}}-\frac{|\left\langle A,\Rmnum{2}\right\rangle |^{2}}{\left(\operatorname{det}A\right)^{\frac{1}{n}}} \quad\quad \text{on}\quad \Sigma,
		\end{equation}
   where $\nabla^{\Sigma}$ denotes the induced Levi-Civita connection on $\Sigma$. Note that $\partial \Sigma=\varnothing$, so we need not to specify  a boundary condition. By standard elliptic PDE theory, there exists a smooth solution $u:\Sigma\longrightarrow \mathbb{R}$ to \eqref{auxilrayPDE} (cf. Theorem 6.30 in \cite{GT}).
   
       As in \cite{B3}, we define a contact set $A_{r}$ for a fixed positive number $r$ as the set of all points $(\bar{x},\bar{y})\in T^{\perp}\Sigma$ with the property that
        \begin{equation}
        ru\left( x\right) +\frac{1}{2}d\left( x,\mathrm{exp}_{\bar{x}}\left( r\nabla^{\Sigma}u\left( \bar{x}\right)+r\bar{y}\right) \right) ^{2}
        \geq ru\left(\bar{x}\right) +\frac{1}{2}r^{2} \left( \left| \nabla^{\Sigma}u\left( \bar{x}\right)\right| ^{2}+\left| \bar{y}\right| ^{2} \right),\;\forall x\in \Sigma.
        \end{equation} 

Let $\Phi_{r}:T^{\perp}\Sigma\longrightarrow M$ be the transport map defined by
    \begin{equation}
    	\Phi_{r}\left( x, y\right) =\mathrm{exp}_{x}\left( r\nabla^{\Sigma}u\left( x\right)+ry\right) ,\quad \forall x\in \Sigma, y\in T^{\perp}_{x}\Sigma.
    \end{equation}

 \begin{lemma}[Lemma 3.2 in \cite{YZ}] \label{containset}
	$M$ is contained in the set $\Phi_{r}\left( A_{r}\right)$, for each $0<r<\infty$.
    \end{lemma}

    \begin{proof}
    For the readers' convenience, we briefly sketch its proof here. For $\forall p\in M$, consider the function $f(x):=ru(x)+\frac{1}{2}d^{2}(x,p)$. Since $\Sigma$ is closed, then $f$ must attain its minimum at some point on $\Sigma$, say $\bar{x}$. Let $\bar{\gamma}$ be a minimizing geodesic on $M$ with $\bar{\gamma}(0)=\bar{x}$ and $\bar{\gamma}(1)=p$. Then one can check that the geodesic $\bar{\gamma}$ also minimizes the functional 
    $$ru(\gamma(0)+\frac{1}{2}\int_{0}^{1}|\gamma^{\prime}(t)|^{2}dt$$ among all smooth paths $\gamma$ with $\gamma(0)\in \Sigma$ and $\gamma(1)=p$. By the first variation formula of energy, we get 
    \[
    r\nabla^{\Sigma}u(\bar{x})-\bar{\gamma}^{\prime}(0)\in T^{\perp}_{\bar{x}}\Sigma.
    \]
    Thus, there exists $\bar{y}\in T^{\perp}_{\bar{x}}\Sigma$ such that 
    \[
     r\nabla^{\Sigma}u(\bar{x})+r\bar{y}=\bar{\gamma}^{\prime}(0).
    \]
    We can easily verify that $(\bar{x},\bar{y})\in A_{r}$, and for this $(\bar{x},\bar{y})$, it holds $\Phi_{r}(\bar{x},\bar{y})=p$.
    \end{proof}

   We also need the following three lemmata in \cite{B3}. 
    
    \begin{lemma}
    	(Lemma 4.3 in \cite{B3})	
     Suppose that $\left( \bar{x},\bar{y}\right)\in A_{r}$ and
     let $\bar{\gamma}\left( t\right):=\mathrm{exp}_{\bar{x}}
     \left( rt\nabla^{\Sigma}u\left( \bar{x}\right)+rt\bar{y}\right)$ 
     for $t\in \left[ 0, 1\right]$. If $Z$ is a vector field along $\bar{\gamma}$ satisfying $Z\left( 0\right)\in T_{x}\Sigma$ and $Z\left( 1\right) =0$, then
    	\begin{equation*}
    	\begin{split}	
         r\left( \nabla^{2}_{\Sigma}u\right) \left(Z\left( 0\right), Z\left( 0\right) \right)-r\left\langle\Rmnum{2}\left(Z\left( 0\right), Z\left( 0\right)\right), \bar{y} \right\rangle \\ +\int_{0}^{1}\left| \nabla_{t}Z\left( t\right) \right| ^{2}
         -R\left(\bar{\gamma}\left( t\right), Z\left( t\right) , \bar{\gamma}\left( t\right), Z\left( t\right) \right) \geq0.
        \end{split}	
    	\end{equation*}		
    \end{lemma}

     \begin{lemma}
    (Lemma 4.4 in \cite{B3})
      For $\left( \bar{x},\bar{y}\right)\in A_{r}$, we have 
      $g+r\operatorname{Hess}^{\Sigma}u\left(\bar{x}\right)-r\left\langle\Rmnum{2}\left( \bar{x}\right),\bar{y}\right\rangle \geq0$.
     \end{lemma}
 
     \begin{lemma}
  	  (Lemma 4.5 in \cite{B3})	
    Suppose that $\left( \bar{x},\bar{y}\right)\in A_{r}$ and
  	let $\bar{\gamma}\left( t\right):=\mathrm{exp}_{\bar{x}}
  	\left( rt\nabla^{\Sigma}u\left( \bar{x}\right)+rt\bar{y}\right)$ 
  	for $t\in \left[ 0, 1\right]$.  
  	Moreover, let $\left\lbrace e_{1}, \ldots, e_{n}\right\rbrace$ be an orthonormal basis of $T_{\bar{x}}\Sigma$.
  	Suppose that $W$ is a Jacobi field along $\bar{\gamma}$ satisfying  $W\left( 0\right)\in T_{\bar{x}}\Sigma$ and
  	$\left\langle \nabla_{t}W\left( 0\right), e_{j}\right\rangle=
  	r\left( \operatorname{Hess}^{\Sigma}u\right) \left(W\left( 0\right), e_{j}\right)-r\left\langle\Rmnum{2}\left(W\left( 0\right), e_{j}\right),\bar{y}\right\rangle$
  	for each $1\leq j\leq n$.
  	If $W\left( \tau\right)=0$ for some $0<\tau<1$, 
  	then $W$ vanishes identically.
     \end{lemma}
    
    The next estimate is key to our proof.
    
    \begin{lemma}\label{keyestimate}
    For $(\bar{x},\bar{y})\in A_{r}$, we have 
    \begin{equation}
		\begin{aligned}
		    \left\langle A(\bar{x}),\operatorname{Hess}^{\Sigma}u(\bar{x})-\langle\Rmnum{2}(\bar{x}),\bar{y}  \rangle \right\rangle&\leq\frac{\left(\operatorname{det}A\right)^{\frac{1}{n}}\left(|\nabla^{\Sigma}u(\bar{x})|^{2}+|\bar{y}|^{2}  \right)}{4}+ \left(\operatorname{det}A\right)^{\frac{1}{n}}\log \left(\operatorname{det}A\right)^{\frac{1}{n}} \\
		    &\quad\; -\frac{\left|2\langle A(\bar{x}),\Rmnum{2}(\bar{x})\rangle+\left(\operatorname{det}A\right)^{\frac{1}{n}}\cdot \bar{y}\right|^{2}}{4\left(\operatorname{det}A\right)^{\frac{1}{n}}}.
		\end{aligned}
    \end{equation}
    \end{lemma}
    
    \begin{proof}
    Note that 
    $$\operatorname{div}_{\Sigma}(A(\nabla^{\Sigma}u))=\langle \operatorname{div}_{\Sigma}A,\nabla^{\Sigma} u\rangle+\langle A,\operatorname{Hess}^{\Sigma}u \rangle.  $$
    Comnining the above identity with the equation of $u$, we get 
    \[
    \begin{aligned}
    &\left\langle A(\bar{x}),\operatorname{Hess}^{\Sigma}u(\bar{x})-\langle\Rmnum{2}(\bar{x}),\bar{y}  \rangle \right\rangle \\
    &=\operatorname{div}_{\Sigma}(A(\nabla^{\Sigma}u))-\langle \operatorname{div}_{\Sigma}A,\nabla^{\Sigma} u\rangle-\langle A, \langle\Rmnum{2},\bar{y}\rangle\rangle\\
    &=\left(\operatorname{det}A\right)^{\frac{1}{n}}\log \left(\operatorname{det}A\right)^{\frac{1}{n}} -\frac{|\operatorname{div }_{\Sigma}A|^{2}}{\left(\operatorname{det}A\right)^{\frac{1}{n}}}-\frac{|\left\langle A,\Rmnum{2}\right\rangle |^{2}}{\left(\operatorname{det}A\right)^{\frac{1}{n}}}-\langle \operatorname{div}_{\Sigma}A,\nabla^{\Sigma} u\rangle-\langle A, \langle\Rmnum{2},\bar{y}\rangle\rangle.
    \end{aligned}
    \] 
    
    By direct calculations, we have
    \begin{equation}\label{dc1}
\begin{aligned}
    &\frac{|\operatorname{div }_{\Sigma}A|^{2}}{\left(\operatorname{det}A\right)^{\frac{1}{n}}}+\langle \operatorname{div}_{\Sigma}A,\nabla^{\Sigma} u\rangle\\
    &=\frac{4|\operatorname{div}_{\Sigma}A|^{2}+4\langle \operatorname{div}_{\Sigma}A, \left(\operatorname{det}A\right)^{\frac{1}{n}}\nabla^{\Sigma} u\rangle}{4\left(\operatorname{det}A\right)^{\frac{1}{n}}}\\
    &=\frac{ 4|\operatorname{div}_{\Sigma}A|^{2}+4\langle \operatorname{div}_{\Sigma}A, \left(\operatorname{det}A\right)^{\frac{1}{n}}\nabla^{\Sigma} u\rangle+|\left(\operatorname{det}A\right)^{\frac{1}{n}}\nabla^{\Sigma} u|^{2}}{4\left(\operatorname{det}A\right)^{\frac{1}{n}}}-\frac{\left(\operatorname{det}A\right)^{\frac{1}{n}}|\nabla^{\Sigma} u|^{2}}{4}\\
    &=\frac{|2\operatorname{div}_{\Sigma}A+\left(\operatorname{det}A\right)^{\frac{1}{n}}\nabla^{\Sigma} u|^{2}}{4\left(\operatorname{det}A\right)^{\frac{1}{n}}}-\frac{\left(\operatorname{det}A\right)^{\frac{1}{n}}|\nabla^{\Sigma} u|^{2}}{4},
\end{aligned}
    \end{equation}
    and 
        \begin{equation}\label{dc2}
\begin{aligned}
&\frac{|\left\langle A,\Rmnum{2}\right\rangle |^{2}}{\left(\operatorname{det}A\right)^{\frac{1}{n}}}+\langle A, \langle\Rmnum{2},\bar{y}\rangle\rangle\\
&=\frac{|\left\langle A,\Rmnum{2}\right\rangle |^{2}}{\left(\operatorname{det}A\right)^{\frac{1}{n}}}+\langle\langle A, \Rmnum{2}\rangle,\bar{y}\rangle\\
&=\frac{|2\langle A,\Rmnum{2}\rangle|^{2}+\langle 4\langle A,\Rmnum{2}\rangle,\left(\operatorname{det}A\right)^{\frac{1}{n}}\cdot \bar{y} \rangle+|\left(\operatorname{det}A\right)^{\frac{1}{n}}\cdot \bar{y}|^{2}}{4\left(\operatorname{det}A\right)^{\frac{1}{n}}}-\frac{\left(\operatorname{det}A\right)^{\frac{1}{n}} |\bar{y}|^{2}}{4}\\
&=\frac{|2\langle A,\Rmnum{2}\rangle+\left(\operatorname{det}A\right)^{\frac{1}{n}}\cdot \bar{y}|^{2}}{4\left(\operatorname{det}A\right)^{\frac{1}{n}}}-\frac{\left(\operatorname{det}A\right)^{\frac{1}{n}} |\bar{y}|^{2}}{4}.
\end{aligned}
    \end{equation}
    
    Combining \eqref{dc1} and \eqref{dc2}, we obtain
    \[
    \begin{aligned}
    	    \left\langle A(\bar{x}),\operatorname{Hess}^{\Sigma}u(\bar{x})-\langle\Rmnum{2}(\bar{x}),\bar{y}  \rangle \right\rangle&\leq\frac{\left(\operatorname{det}A\right)^{\frac{1}{n}}\left(|\nabla^{\Sigma}u(\bar{x})|^{2}+|\bar{y}|^{2}  \right)}{4}+ \left(\operatorname{det}A\right)^{\frac{1}{n}}\log \left(\operatorname{det}A\right)^{\frac{1}{n}} \\
		    &\quad\; -\frac{\left|2\langle A(\bar{x}),\Rmnum{2}(\bar{x})\rangle+\left(\operatorname{det}A\right)^{\frac{1}{n}}\cdot \bar{y}\right|^{2}}{4\left(\operatorname{det}A\right)^{\frac{1}{n}}}.
    \end{aligned}
    \]
    \end{proof}
    
    \begin{lemma}[Lemma 3.6 in \cite{B4}]\label{Brendle3}
    The Jacobian determinant of $\Phi_{r}$ satisfies
    \[
    |\operatorname{det} D \Phi_{r}(\bar{x}, \bar{y})| \leq r^m \operatorname{det}\left(g+r \operatorname{Hess}^{\Sigma}u(\bar{x})-r\langle \Rmnum{2}(\bar{x}), \bar{y}\rangle\right)
    \]
    for all $(\bar{x},\bar{y})\in A_{r}$.
    \end{lemma}
    
    \begin{lemma}\label{detestimate}
    For $(\bar{x},\bar{y})\in A_{r}$, the Jacobian determinant of $\Phi_{r}$ can be estimated as follows:
    \begin{equation}
    |\operatorname{det}D\Phi_{r}(\bar{x},\bar{y})|\leq r^{n+m}\left(\operatorname{det}A\right)^{\frac{1}{n}}e^{\frac{C_{2}}{C_{1}r}-n}e^{\frac{d^{2}(x,\Phi_{r}(\bar{x},\bar{y}))}{4r^{2}}}e^{-\frac{|2\xi(\bar{x})+\bar{y}|^{2}}{4}},
    \end{equation}
    where $\xi(\bar{x}):=\frac{1}{\left(\operatorname{det}A(\bar{x})\right)^{\frac{1}{n}}}\langle A(\bar{x}),\Rmnum{2}(\bar{x}) \rangle\in T_{\bar{x}}^{\perp}\Sigma$, $C_{1}$ and $C_{2}$ are positive constants depending only on $A$.
    \end{lemma}
    
    \begin{proof}
    We start with Lemma \ref{Brendle3}, 
    \[
    \begin{aligned}
     |\operatorname{det} D \Phi_{r}(\bar{x}, \bar{y})| &\leq r^m \operatorname{det}\left(g+r \operatorname{Hess}^{\Sigma}u(\bar{x})-r\langle \Rmnum{2}(\bar{x}), \bar{y}\rangle\right)\\
     &=r^{n+m}\frac{1}{\operatorname{det}A(\bar{x})}\operatorname{det}\left(  A(\bar{x})\circ\left( \frac{g}{r}+\operatorname{Hess}^{\Sigma}u(\bar{x})-\langle \Rmnum{2}(\bar{x}), \bar{y}\rangle  \right) \right)\\
     &\leq r^{n+m}\frac{1}{\operatorname{det}A(\bar{x})}\left(  \frac{ \operatorname{tr}\left[A(\bar{x})\circ \left(\frac{g}{r}+\operatorname{Hess}^{\Sigma}u(\bar{x})-\langle \Rmnum{2}(\bar{x}), \bar{y}\rangle    \right)  \right]   }{n}  \right)^{n}\\
     &=r^{n+m}\frac{1}{\operatorname{det}A(\bar{x})}\left( \frac{\langle A,\frac{g}{r} \rangle}{n}+\frac{\langle A(\bar{x}), \operatorname{Hess}^{\Sigma}u(\bar{x})-\langle \Rmnum{2}(\bar{x}), \bar{y}\rangle\rangle}{n}  \right)^{n}\\
     &=r^{n+m}\left( \frac{ \frac{1}{\left(\operatorname{det}A\right)^{\frac{1}{n}}}\langle A,\frac{g}{r} \rangle+\frac{1}{\left(\operatorname{det}A\right)^{\frac{1}{n}}} \langle A, \operatorname{Hess}^{\Sigma}u(\bar{x})-\langle \Rmnum{2}(\bar{x}), \bar{y}\rangle \rangle }{n} \right)^{n},
    \end{aligned}
    \]
    where we have used the arithmetic-geometric mean inequality $\mathrm{det}AB \leq\left(\frac{\mathrm{tr}\left( AB\right)}{n} \right)^{n}$.

    By the elementary inequality $\lambda \leq e^{\lambda -1}$, we get 
    \[
    \begin{aligned}
     |\operatorname{det} D \Phi_{r}(\bar{x}, \bar{y})| &\leq r^{n+m} e^{ \frac{1}{\left(\operatorname{det}A\right)^{\frac{1}{n}}}\langle A,\frac{g}{r} \rangle+\frac{1}{\left(\operatorname{det}A\right)^{\frac{1}{n}}} \langle A, \operatorname{Hess}^{\Sigma}u(\bar{x})-\langle \Rmnum{2}(\bar{x}), \bar{y}\rangle \rangle -n}.
    \end{aligned}
    \]
    Applying Lemma \ref{keyestimate}, then we obtain 
    \begin{equation}\label{det1}
    \begin{aligned}
     |\operatorname{det} D \Phi_{r}(\bar{x}, \bar{y})| &\leq r^{n+m} e^{ \frac{\operatorname{tr}_{g}(A)}{r\cdot\left(\operatorname{det}A\right)^{\frac{1}{n}}}+\log\left(\operatorname{det}A\right)^{\frac{1}{n}}+\frac{|\nabla^{\Sigma}u|^{2}+|\bar{y}|^{2}}{4}-\frac{|2\xi(\bar{x})+\bar{y}|^{2}}{4}-n}.
    \end{aligned}
    \end{equation}
    
    By assumption, $A$ is uniformly positive definite, thus there exists a positive constant $C_{1}$ such that 
    \begin{equation}\label{det2}
    \left(\operatorname{det}A\right)^{\frac{1}{n}}\geq C_{1}>0,\quad \forall x\in \Sigma.
    \end{equation}
    On the other hand, since $\Sigma$ is compact, the smooth function $\operatorname{tr}_{g}(A)(x)$ must be bounded by some positive constants $C_{2}$, that is,
    \begin{equation}\label{det3}
    \operatorname{tr}_{g}(A)(x)\leq C_{2},\quad \forall x\in\Sigma.
    \end{equation}
    And one can easily see from the proof of Lemma \ref{containset} that 
    \begin{equation}\label{det4}
    d\left(\bar{x},\Phi_{r}(\bar{x},\bar{y})\right)^{2}=r^{2}\left(|\nabla^{\Sigma}u(\bar{x})|^{2}+|\bar{y}|^{2}\right),\quad \forall (\bar{x},\bar{y})\in A_{r}.
    \end{equation}

    Finally, combining \eqref{det1}, \eqref{det2}, \eqref{det3} and \eqref{det4}, we conclude that
        \begin{equation*}
    |\operatorname{det}D\Phi_{r}(\bar{x},\bar{y})|\leq r^{n+m}\left(\operatorname{det}A\right)^{\frac{1}{n}}e^{\frac{C_{2}}{C_{1}r}-n}e^{\frac{d^{2}(\bar{x},\Phi_{r}(\bar{x},\bar{y}))}{4r^{2}}}e^{-\frac{|2\xi(\bar{x})+\bar{y}|^{2}}{4}}.
    \end{equation*}
 This completes the proof.
    \end{proof}
    
    With the preparations above, we are now in a position to prove Theorem\ref{main1}.\\
    
    \noindent\textbf{\emph{Proof of Theorem \ref{main1}.}} By Lemma \ref{containset}, we can assign each point $p\in M$ some point $(x_{p},y_{p})\in A_{r}$ with the property $\Phi_{r}(x_{p},y_{p})=p$. Then by Lemma \ref{containset} and area formula originating from geometric measure theory, we obtain
    \begin{equation}\label{area1}
    \begin{aligned}
    \int_{M}e^{-\frac{d(x_{p},p)^{2}}{4r^{2}}}dV(p)&\leq \int_{M}\left(  \int_{(x,y)\in A_{r}|\Phi_{r}(x,y)=p}e^{-\frac{d(x,\Phi_{r}(x,y))^{2}}{4r^{2}}} d\mathcal{H}^{0}\right)dV(p),
    \end{aligned}
    \end{equation}
    where $\mathcal{H}^{0}$ is the counting measure. We then apply area formula with the map $\Phi_{r}:A_{r}\longrightarrow M$ to the right hand side of \eqref{area1} to get 
        \begin{equation}\label{area2}
    \begin{aligned}
    \int_{M}e^{-\frac{d(x_{p},p)^{2}}{4r^{2}}}dV(p)&\leq \int_{A_{r}}e^{-\frac{d(x,\Phi_{r}(x,y))^{2}}{4r^{2}}}|\operatorname{det}D\Phi_{r}(x,y)|dV(x,y)\\
    &=\int_{\Sigma}\left( \int_{T_{x}^{\perp}\Sigma} e^{-\frac{d(x,\Phi_{r}(x,y))^{2}}{4r^{2}}}|\operatorname{det}D\Phi_{r}(x,y)| \chi_{A_{r}}(x,y)dy\right)dV(x),
    \end{aligned}
    \end{equation}
    where $\chi_{A_{r}}$ is the indicator function. Using the estimate in Lemma \ref{detestimate} to the intergrand of right hand side of \eqref{area2} yields the following
            \begin{equation}\label{area3}
    \begin{aligned}
    \int_{M}e^{-\frac{d(x_{p},p)^{2}}{4r^{2}}}dV(p)&\leq \int_{\Sigma}\left(\int_{T_{x}^{\perp}\Sigma} r^{n+m}\left(\operatorname{det}A\right)^{\frac{1}{n}}e^{\frac{C_{2}}{C_{1}r}-n}e^{-\frac{|2\xi(x)+y|^{2}}{4}}\chi_{A_{r}}(x,y)dy \right)dV(x)\\
    &\leq \int_{\Sigma}\left(\int_{T_{x}^{\perp}\Sigma} r^{n+m}\left(\operatorname{det}A\right)^{\frac{1}{n}}e^{\frac{C_{2}}{C_{1}r}-n}e^{-\frac{|2\xi(x)+y|^{2}}{4}}dy \right)dV(x)\\
    &=r^{n+m}e^{\frac{C_{2}}{C_{1}r}-n}\int_{\Sigma}\left(\operatorname{det}A\right)^{\frac{1}{n}}\left( \int_{T_{x}^{\perp}\Sigma}e^{-\frac{|2\xi(x)+y|^{2}}{4}}dy \right)dV(x)\\
    &=r^{n+m}e^{\frac{C_{2}}{C_{1}r}-n}(4\pi)^{\frac{m}{2}}\int_{\Sigma}\left(\operatorname{det}A(x)\right)^{\frac{1}{n}}dV(x).
    \end{aligned}
    \end{equation}
       In the last equality, we have used the identity of Guassian function $(4\pi)^{\frac{N}{2}}=\int_{\mathbb{R}^{N}}e^{-\frac{|x|^{2}}{4}}dx$.
       
       Dividing both sides by $r^{n+m}$ in \eqref{area3}, and letting $r\to +\infty$, we then obtain
       \[
       (4\pi)^{\frac{m+n}{2}}\theta\leq e^{-n}(4\pi)^{\frac{m}{2}}\int_{\Sigma}\left(\operatorname{det}A(x)\right)^{\frac{1}{n}}dV(x).
       \]
       Following by taking logarithmic fubction, multiplying the inequality with $\left(\operatorname{det}A\right)^{\frac{1}{n}}$ and then integrating over $\Sigma$, we can get 
       \begin{equation}
       \int_{\Sigma}\left(\operatorname{det}A\right)^{\frac{1}{n}}\left(n+\frac{n}{2}\log(4\pi)+\log\theta\right)\leq \int_{\Sigma}\left(\operatorname{det}A\right)^{\frac{1}{n}}\log\left( \int_{\Sigma}\left(\operatorname{det}A\right)^{\frac{1}{n}}\right).
       \end{equation}
       This means that the above inequality is exactly \eqref{goal}. Therefore, in this special case that $\Sigma$ is connected, Theorem \ref{main1} has been proved.
    
    It remains to show the case when $\Sigma$ is disconnected. One can follow Brendle's idea \cite{B3} to apply the inequality to each individual connected component of $\Sigma$, and sum over all connected components with the help of following simple inequality
    \[
    a\log a+b\log b<a\log (a+b)+b\log (a+b)=(a+b)\log(a+b).
    \]
     This completes our proof of Theorem \ref{main1}.   \qed
    
    \section{Proof of Theorem \ref{main2}}
    
    In this section, we prove a Sobolev-type inequality for uniformly positive symmetric 2-tensor fields on smooth domains in Riemannian manifolds. Throughout this section, we assume that $(M,g)$ is a complete, noncompact Riemannian manifold of dimension $n$ with nonnegative sectional curvature. Let $D$ be a compact domain in $M$ and $A$ be a smooth symmetric uniformly positive definite $\left( 0,2\right)$-tensor field on $D$.
    
    As in \cite{B3}, it suffices to prove the assertion in the special case when $D$ is connected. Due to the scaling invariant property of \eqref{sobolev}, we can make the following normalization 
    \begin{equation}
    \int_{D}|\operatorname{div }A|+\int_{\partial D}|A(\nu)|=n\int_{D}\left(\operatorname{det}A\right)^{\frac{1}{n-1}}.
    \end{equation}
    Then we only need to prove 
        \begin{equation}\label{goal2}
    \int_{D}\left(\operatorname{det}A\right)^{\frac{1}{n-1}}\geq |B^{n}|\theta.
    \end{equation}
    
    Since $D$ is connected and $A$ is positive definite and symmetric,  there exists a solution $u:D\longrightarrow \mathbb{R}$ to the following Neumann boundary problem:
    \begin{equation}
    \begin{cases}\operatorname{div}\left(A(\nabla u)\right)(x)=n\left(\operatorname{det}A(x)\right)^{\frac{1}{n-1}}-|\operatorname{div }A(x)|, & \text { in } D\\ \left\langle A(\nabla u)(x), \nu(x)\right\rangle=|A(\nu(x))|, & \text { on } \partial D.\end{cases}
    \end{equation}
    where $\nu$ is the outward unit normal to $\partial D$. By standard elliptic regularity theory, $u$ is in the class $C^{2,\alpha}(D)$ for every
$0<\alpha<1$ (cf. Theorem 6.30 in \cite{GT}).

As in \cite{B3}, we define
\[
U:=\{x \in D \backslash \partial D:|\nabla u(x)|<1\}.
\]
For each $r>0$, we denote by $A_{r}$ the contact set, that is the set of all points $\bar{x}\in U$ with the property that
\[
r u(x)+\frac{1}{2} d\left(x, \exp _{\bar{x}}(r \nabla u(\bar{x}))\right)^2 \geq r u(\bar{x})+\frac{1}{2} r^2|\nabla u(\bar{x})|^2,\quad \forall x\in D.
\]
Moreover, for each $r>0$, we define the transport map $\Phi_{r}:D\longrightarrow M$ by
\begin{equation}
\Phi_r(x)=\exp _x(r \nabla u(x))
\end{equation}
for all $x\in D$. Note that for each $0<\alpha<1$, since $u$ is of class $C^{2,\alpha}$, $\Phi_{r}$ is of class $C^{1,\alpha}$.

\begin{lemma}\label{hessestimate}
For $x\in U$, we have $\langle A(x),\operatorname{Hess}u(x) \rangle\leq n\left(\operatorname{det}A(x)\right)^{\frac{1}{n-1}}$.
\end{lemma}

\begin{proof}
Using the inequality $|\nabla u(x)|<1$ and the Cauchy-Schwarz inequality, we obtain
$$
-\langle\operatorname{div }A(x), \nabla u(x)\rangle \leq|\operatorname{div }A(x)|.
$$
Moreover, $\operatorname{div}\left(A(\nabla u)\right)=n\left(\operatorname{det}A\right)^{\frac{1}{n-1}}-|\operatorname{div }A|$ by definition of $u$, together with the identity
$$\operatorname{div}(A(\nabla u))=\langle \operatorname{div}A,\nabla u\rangle+\langle A,\operatorname{Hess}u \rangle $$
imply
\[
\begin{aligned}
\langle A(x),\operatorname{Hess}u(x) \rangle &=\operatorname{div}(A(\nabla u))(x)-\langle \operatorname{div}A(x),\nabla u(x)\rangle\\
&=n\left(\operatorname{det}A(x)\right)^{\frac{1}{n-1}}-|\operatorname{div }A(x)|-\langle \operatorname{div}A(x),\nabla u\rangle\\
&\leq n\left(\operatorname{det}A(x)\right)^{\frac{1}{n-1}}.
\end{aligned}
\]
\end{proof}

\begin{lemma}\label{minimum}
 Fix a point $p\in M$ with the property that $d(x,p)<r$ for all $x\in D$. Then the function $f(x):=ru(x)+\frac{1}{2}d(x,p)^2$ cannot attain its minimum on $\partial D$. 
\end{lemma}

\begin{proof}
We adapt the argument in Proposition 3.2 in \cite{WYZ}. Since $A$ is positive definite and symmetric on $\overline{D}$, we have  
\[
\left\langle \nabla u, \frac{A(\nu) }{|A(\nu) |}\right\rangle =\frac{1 }{|A(\nu) |}\langle  A(\nabla u),\nu\rangle=1
\]
and
\[
\left\langle \nu, \frac{A(\nu) }{|A(\nu) |}\right\rangle=\frac{1 }{|A(\nu) |}\langle  A(\nu),\nu\rangle>0.
\]
This means $A(\nu)$ is also outward-pointing.

Assume that $x_{0}\in \partial D$ is a minimum point of $f$. Take a curve $\gamma(t)$ with $\gamma(0)=x_{0}$ and $\gamma^{\prime}(0)=-\frac{A(\nu) }{|A(\nu) |}(x_{0})$. If $x_{0}$ is not in the cut locus of $p$, then by direct calculations, we have
\[
\begin{aligned}
\langle \nabla f,\gamma^{\prime}(0)\rangle&=r\langle \nabla u,\gamma^{\prime}(0) \rangle+d(x,p)\langle \nabla d,\gamma^{\prime}(0)\rangle\\
&\leq -r+d(x,p)\\
&<0.
\end{aligned}
\]
If $x_{0}$ lies in the cut locus of $p$, we take a minimal geodesic $C$ connecting $x_{0}$ and $p$ and then move the point $p$ to a new point $q$ along $C$. It is obvious that $C$ is also the minimal geodesic connecting $q$ and $x_{0}$. Now the distance function $\tilde{d}:=d(\cdot,q)$ is differentiable at $x_{0}$. We have
\[
\begin{aligned}
\langle \nabla f,\gamma^{\prime}(0)\rangle&=\left.\frac{d}{dt}\right|_{t=0}\left[\left( ru(x)+\frac{1}{2}(d(x,q)+d(q,p))^{2}  \right)\circ\gamma(t)\right]\\
&=r\langle \nabla u,\gamma^{\prime}(0) \rangle+(d(x,q)+d(q,p))\langle \nabla \tilde{d},\gamma^{\prime}(0)\rangle\\
&=r\langle \nabla u,\gamma^{\prime}(0) \rangle+d(x,p)\langle \nabla \tilde{d},\gamma^{\prime}(0)\rangle\\
&\leq -r+d(x,p)\\
&<0.
\end{aligned}
\]
Both cases will lead to a contradiction. Thus $f$ cannot attain its minimum on $\partial D$. 
\end{proof}

\begin{lemma}\label{caot}
The set
$$
\{p \in M: d(x, p)<r \text { for all } x \in D\}
$$
is contained in the set $\Phi_{r}(A_{r})$.
\end{lemma}

\begin{proof}
Take any $p\in M$ with the property that $d(x,p)<r$ for all $x\in D$. By Lemma \ref{minimum}, there is $\bar{x}\in D\setminus \partial D$ where the function $ru(x)+\frac{1}{2}d(x,p)^2$ attains its minimum. Let $\bar{\gamma}$ be a minimizing geodesic on $M$ with $\bar{\gamma}(0)=\bar{x}$ and $\bar{\gamma}(1)=p$. Then one can check that the geodesic $\bar{\gamma}$ also minimizes the functional 
    $$ru(\gamma(0)+\frac{1}{2}\int_{0}^{1}|\gamma^{\prime}(t)|^{2}dt$$ among all smooth paths $\gamma$ with $\gamma(0)\in D$ and $\gamma(1)=p$. By the first variation formula of energy, we get 
    \[
    r\nabla u(\bar{x})=\bar{\gamma}^{\prime}(0).
    \]
    From this, we deduce that $\Phi_{r}(\bar{x})=p$. Moreover,
    \[
    r|\nabla u(\bar{x})|=|\bar{\gamma}^{\prime}(0)|=d(\bar{x},p).
    \]
    By assumption, $d(\bar{x},p)<r$. This implies $|\nabla u(\bar{x})|<1$ and hence $\bar{x}\in U$. At last, 
    \[
    \begin{aligned}
r u(x)+\frac{1}{2} d\left(x, \exp _{\bar{x}}(r \nabla u(\bar{x}))\right)^2 & =r u(x)+\frac{1}{2} d(x, p)^2 \\
& \geq r u(\bar{x})+\frac{1}{2} d(\bar{x}, p)^2 \\
& =r u(\bar{x})+\frac{1}{2} r^2|\nabla u(\bar{x})|^2.
\end{aligned}
    \]
    This means $\bar{x}\in A_{r}$.
\end{proof}

\begin{lemma}[Lemma 2.3 in \cite{B3}]\label{variation2}
Suppose that $\bar{x} \in A$, let $\bar{\gamma}(t):=\exp _{\bar{x}}(r t \nabla u(\bar{x}))$ for $t \in$ $[0,1]$. If $Z$ is a vector field along $\bar{\gamma}$ satisfying $Z(1)=0$, then
$$
r\operatorname{Hess} u(Z(0), Z(0))+\int_0^1\left(\left|\nabla_t Z(t)\right|^2-R\left(\bar{\gamma}^{\prime}(t), Z(t), \bar{\gamma}^{\prime}(t), Z(t)\right)\right) d t \geq 0.
$$
\end{lemma}

\begin{lemma}\label{wuguan}
  	  (Lemma 2.4 in \cite{B3})	
    Suppose that $ \bar{x}\in A_{r}$, and
  	let $\bar{\gamma}\left( t\right):=\mathrm{exp}_{\bar{x}}
  	\left( rt\nabla u\left( \bar{x}\right)\right)$ 
  	for $t\in \left[ 0, 1\right]$.  
  	Moreover, let $\left\lbrace e_{1}, \ldots, e_{n}\right\rbrace$ be an orthonormal basis of $T_{\bar{x}}M$.
  	Suppose that $W$ is a Jacobi field along $\bar{\gamma}$ satisfying 
  	$\left\langle \nabla_{t}W\left( 0\right), e_{j}\right\rangle=
  	r\left( \operatorname{Hess}u\right) \left(W\left( 0\right), e_{j}\right)$
  	for each $1\leq j\leq n$.
  	If $W\left( \tau\right)=0$ for some $0<\tau<1$, 
  	then $W$ vanishes identically.
     \end{lemma}

\begin{lemma}\label{positived}
For $\bar{x}\in A_{r}$, we have $g+r\operatorname{Hess} u(\bar{x})\geq 0$.
\end{lemma}

\begin{proof}
We define $\bar{\gamma}(t):=\exp _{\bar{x}}(r t \nabla u(\bar{x}))$. Fix an arbitrary vector $v\in T_{\bar{x}}M$, and let $V(t)$ denote the parallel transport of $v$ along $\bar{\gamma}$. Applying Lemma \ref{variation2} to the vector field $Z(t):=(1-t)V(t)$ yields
\[
g(v,v)+r\operatorname{Hess} u(v,v)-\int_0^1(1-t)^2 R\left(\bar{\gamma}^{\prime}(t), V(t), \bar{\gamma}^{\prime}(t), V(t)\right) d t \geq 0.
\]
By assumption, $M$ has nonnegative sectional curvature, we obtain $g+r\operatorname{Hess} u(\bar{x})\geq 0$.
\end{proof}

\begin{lemma}\label{detest2}
The Jacobian determinant of $\Phi_r$ is given by
$$
\left|\operatorname{det} D \Phi_r(\bar{x})\right| \leq \operatorname{det} \left(  g+r\operatorname{Hess} u(\bar{x}) \right)
$$
for all $\bar{x} \in A_r$.
\end{lemma}

\begin{proof}
For any $r>0$ and $\bar{x}\in A_{r}$, let us choose an orthonormal basis $\{e_{1},\ldots,e_{n}\}$ of $T_{\bar{x}}M$ such that 
\[
\operatorname{Hess} u(e_{i},e_{j})
\]
is diagonal. Let $\{x_{1},\ldots,x_{n}\}$ be a system of geodesic normal coordinates around $\bar{x}$ such that $\frac{\partial}{\partial x_{i}}=e_{i}$. We denote by $E_{i}(t)$ the parallel transport of $e_{i}$ along $\bar{\gamma}$ for each $1\leq i\leq n$.

Define a geodesic $\bar{\gamma}(t):=\exp _{\bar{x}}(r t \nabla u(\bar{x}))$ for $t\in [0,1]$. For each $1\leq i\leq n$, we consider the unique Jacobi field $X_{i}(t)$ along $\bar{\gamma}$ with initial conditions
\[
X_{i}(0)=e_{i},\quad \left\langle \nabla_t X_i(0), e_j\right\rangle=r\operatorname{Hess} u(e_{i},e_{j}).
\]
Lemma \ref{wuguan} tells us that $\{X_{1}(t),\ldots, X_{n}(t)\}$ are linearly
independent for each $t\in (0,1)$.

Then we define two $n\times n$-matrix $P(t)=(P_{ij}(t))$ and $S(t)=(S_{ij}(t))$ by 
\[
P_{i j}(t):=\left\langle X_i(t), E_j(t)\right\rangle,\quad S_{i j}(t):=R\left(\bar{\gamma}^{\prime}(t), E_i(t), \bar{\gamma}^{\prime}(t), E_j(t)\right)
\]
for $1\leq i,j\leq n$. By the definition of Jacobi fields, we have Jacobi equation
\[
P^{\prime \prime}(t)=-P(t) S(t),\quad \forall t\in (0,1),
\]
 with initial values
 \[
 P_{i j}(0)=\delta_{i j},\quad P_{i j}^{\prime}(0)=r\operatorname{Hess} u(e_{i},e_{j}).
 \]
Since $\{X_{1}(t),\ldots, X_{n}(t)\}$ are linearly independent, $P(t)$ is invertible for each $t\in (0,1)$. We then define $Q(t):=P^{-1}(t)P^{\prime}(t)$, which satisfies the Riccati equation 
\[
Q^{\prime}(t)=-S(t)-Q^2(t),\quad \forall t\in (0,1),
\]
 with initial values
 \[
 Q_{i j}(0)=r\operatorname{Hess} u(e_{i},e_{j}).
 \]
 Moreover, the asymptotic expansion of $Q(t)$ satisfies
 \begin{equation}\label{asy}
 Q_{i j}(t)=r\operatorname{Hess} u(e_{i},e_{j})+O(t),\quad \text{as}\; t\to 0.
 \end{equation}
By assumption, $r\operatorname{Hess} u(e_{i},e_{j})$ is diagonal, thus we can write
\[
r\operatorname{Hess} u(e_{i},e_{j})=r\lambda_{i}\delta_{ij}.
\]

Since $M$ has nonnegative sectional curvature, we know that $S(t)\geq 0$ and hence 
\begin{equation}
Q^{\prime}(t)\leq -Q^{2}(t)
\end{equation}
for all $t\in (0,1)$. From \eqref{asy}, we know $Q_{i j}(s)=r\lambda_{i}\delta_{ij}+O(s)$ as $s\to 0$. Therefore we can find a small number $t_{0}>0$ such that 
\[
Q_{i j}(s)<r(\lambda_{i}+\sqrt{s})\delta_{ij}
\]
for $0<s<t_{0}$. It follows from Lemma \ref{positived} that $1+r\lambda_{i}\geq 0$. In particular, $1+rt\lambda_{i}> 0$ for each $1\leq i\leq n$ and all $t\in (0,1)$. By standard ODE comparison principle, we obtain
\[
Q_{ij}(t)\leq \frac{r(\lambda_{i}+\sqrt{s})\delta_{ij}}{1+r(t-s)(\lambda_{i}+\sqrt{s})}
\]
for all $s\in (0,t_{0})$ and all $t\in [s,1)$. Letting $s\to 0$, we get 
\begin{equation}
Q_{ij}(t)\leq \frac{r\lambda_{i}\delta_{ij}}{1+rt\lambda_{i}}
\end{equation}
for each $t\in (0,1)$. By taking trace, we obtain
\begin{equation}
\operatorname{tr}Q(t)\leq \sum_{i=1}^{n}\frac{r\lambda_{i}}{1+rt\lambda_{i}}.
\end{equation}

Finally, we can estimate $\left|\operatorname{det} D \Phi_r(\bar{x})\right|$. Observe that 
\[
\frac{\partial \Phi_{r}}{\partial x_{i}}(\bar{x})=X_{i}(1),\quad \forall 1\leq i\leq n.
\]
Hence $\left|\operatorname{det} D \Phi_r(\bar{x})\right|=|\operatorname{det} P(1)|$. Note that $\operatorname{det} P(t)>0$ when $t$
is sufficiently small. On the other hand, since $P(t)$ is invertible for all $t\in (0,1)$, it implies $\operatorname{det} P(t)>0$ for all $t\in (0,1)$.

 From $Q(t)=P^{-1}(t)P^{\prime}(t)$, we get 
\[
\frac{d}{dt}\log\operatorname{det}P(t)=\operatorname{tr}Q(t)\leq \sum_{i=1}^{n}\frac{r\lambda_{i}}{1+rt\lambda_{i}}.
\]
Integrating the above inequality over $[\varepsilon,t]$ for $0<\varepsilon<t$ and letting $\varepsilon\to 0^{+}$, we obtain
\[
\operatorname{det}P(t)\leq \prod_{i=1}^{n}(1+r\lambda_{i}t),
\]
where we have used $P(0)=I_{n}$. Sending $t\to 1$, we get 
\[
\left|\operatorname{det} D \Phi_r(\bar{x})\right|=\operatorname{det} P(1)\leq \prod_{i=1}^{n}(1+r\lambda_{i})=\operatorname{det} \left(  g+r\operatorname{Hess} u(\bar{x}) \right).
\]
This completes the proof.
\end{proof}

\begin{lemma}\label{refined}
The Jacobian determinant of $\Phi_r$ is satisfies
$$
\left|\operatorname{det} D \Phi_r(\bar{x} )\right| \leq \frac{1}{\operatorname{det}(\bar{x} )}\left( \frac{\operatorname{tr}_{g}A(\bar{x} )}{n}+r\left(\operatorname{det}A(\bar{x} )\right)^{\frac{1}{n-1}} \right)^{n}
$$
for all $\bar{x} \in A_r$.
\end{lemma}

\begin{proof}
Combining Lemma \ref{hessestimate}, Lemma \ref{detest2} and the arithmetic mean inequality, we obtain
\[
    	\begin{aligned}
    		\left|\operatorname{det} D \Phi_r(\bar{x} )\right| &\leq \operatorname{det} \left(  g+r\operatorname{Hess} u(\bar{x} ) \right)\\
    		=& \dfrac{1}{\mathrm{det}A\left( \bar{x} \right)}\mathrm{det}
    		\left[A\left( \bar{x} \right) \circ \left(g+r\operatorname{Hess} u\left( \bar{x} \right) \right)\right] \\ 
    		\leq& \dfrac{1}{\mathrm{det}A\left( \bar{x} \right)}
    		\left(\dfrac{\mathrm{tr}\left( A\left( \bar{x} \right) \left(g+r\operatorname{Hess} u\left( \bar{x} \right) \right)\right) }{n}\right)^{n}\\
    		=& \dfrac{1}{\mathrm{det}A\left( \bar{x} \right)}
    		\left(\dfrac{\left\langle A\left( \bar{x} \right), g\right\rangle }{n}+\dfrac{ r\left\langle A\left( \bar{x} \right), \operatorname{Hess} u\left( \bar{x} \right) \right\rangle}{n}\right)^{n}\\
    		\leq&\dfrac{1}{\mathrm{det}A\left( \bar{x} \right)}
    		\left(\dfrac{\mathrm{tr}_{g}A\left( \bar{x} \right)}{n}+r\left( \mathrm{det}A\left( \bar{x} \right) \right) ^{\frac{1}{n-1}}\right)^{n}.
    	\end{aligned}
    \]
\end{proof}

  With the above estimates, we now give the proof of Theorem\ref{main2}.\\
    
\noindent\textbf{\emph{Proof of Theorem \ref{main2}.}} By Lemma \ref{caot} and Lemma \ref{refined}, we have
\[
\begin{aligned}
\left|  \{p\in M|d(x,p)<r\text{ for all }x\in D \}\right|&\leq \int_{A_{r}}\left|\operatorname{det} D \Phi_r(\bar{x} )\right|dV(\bar{x})\\
&\leq \int_{U}\frac{1}{\operatorname{det}(x)}\left( \frac{\operatorname{tr}_{g}A(x)}{n}+r\left(\operatorname{det}A(x)\right)^{\frac{1}{n-1}} \right)^{n}dV(x)
\end{aligned}
\]
for all $r>0$. Dividing both sides by $r^{n}$ and letting $r\to +\infty$, we obtain
\[
|B^{n}|\theta\leq \int_{U}\left(\operatorname{det}A(x)\right)^{\frac{1}{n-1}}\leq \int_{D}\left(\operatorname{det}A(x)\right)^{\frac{1}{n-1}}.
\]
This is exactly our goal \eqref{goal2}. Thus we have proved Theorem \ref{main2}.   \qed

\end{document}